\documentclass[11pt]{amsart}

\usepackage{etex}
\usepackage{amsmath}
\usepackage{amssymb}
\usepackage{amscd}
\usepackage{amsthm}
\usepackage{hyperref}
\usepackage{colortbl}

\usepackage{mathtools}
\usepackage{extarrows}

\usepackage[all]{xy}

\usepackage{tikz}
\usepackage{pgfplots}
\usepackage{tikz-cd}

\newtheorem{thm}{Theorem}[section]

\newtheorem{cor}[thm]{Corollary}
\newtheorem{conj}[thm]{Conjecture}
\newtheorem{prop}[thm]{Proposition}

\theoremstyle{remark}

\theoremstyle{definition}

\numberwithin{equation}{section}

\allowdisplaybreaks[4]

\begin{document}

\vfuzz0.5pc
\hfuzz0.5pc 

\newcommand{\claimref}[1]{Claim \ref{#1}}
\newcommand{\thmref}[1]{Theorem \ref{#1}}
\newcommand{\propref}[1]{Proposition \ref{#1}}
\newcommand{\lemref}[1]{Lemma \ref{#1}}
\newcommand{\coref}[1]{Corollary \ref{#1}}
\newcommand{\remref}[1]{Remark \ref{#1}}
\newcommand{\conjref}[1]{Conjecture \ref{#1}}
\newcommand{\questionref}[1]{Question \ref{#1}}
\newcommand{\defnref}[1]{Definition \ref{#1}}
\newcommand{\secref}[1]{\S \ref{#1}}
\newcommand{\ssecref}[1]{\ref{#1}}
\newcommand{\sssecref}[1]{\ref{#1}}

\newcommand{\RED}{{\mathrm{red}}}
\newcommand{\tors}{{\mathrm{tors}}}
\newcommand{\eq}{\Leftrightarrow}

\newcommand{\mapright}[1]{\smash{\mathop{\longrightarrow}\limits^{#1}}}
\newcommand{\mapleft}[1]{\smash{\mathop{\longleftarrow}\limits^{#1}}}
\newcommand{\mapdown}[1]{\Big\downarrow\rlap{$\vcenter{\hbox{$\scriptstyle#1$}}$}}
\newcommand{\smapdown}[1]{\downarrow\rlap{$\vcenter{\hbox{$\scriptstyle#1$}}$}}

\newcommand{\A}{{\mathbb A}}
\newcommand{\I}{{\mathcal I}}
\newcommand{\J}{{\mathcal J}}
\newcommand{\CO}{{\mathcal O}}
\newcommand{\C}{{\mathcal C}}
\newcommand{\BC}{{\mathbb C}}
\newcommand{\BQ}{{\mathbb Q}}
\newcommand{\m}{{\mathcal M}}
\newcommand{\h}{{\mathcal H}}
\newcommand{\Z}{{\mathcal Z}}
\newcommand{\BZ}{{\mathbb Z}}
\newcommand{\W}{{\mathcal W}}
\newcommand{\Y}{{\mathcal Y}}
\newcommand{\T}{{\mathcal T}}
\newcommand{\BP}{{\mathbb P}}
\newcommand{\CP}{{\mathcal P}}
\newcommand{\G}{{\mathbb G}}
\newcommand{\BR}{{\mathbb R}}
\newcommand{\D}{{\mathcal D}}
\newcommand{\DD}{{\mathcal D}}
\newcommand{\LL}{{\mathcal L}}
\newcommand{\f}{{\mathcal F}}
\newcommand{\E}{{\mathcal E}}
\newcommand{\BN}{{\mathbb N}}
\newcommand{\N}{{\mathcal N}}
\newcommand{\K}{{\mathcal K}}
\newcommand{\R}{{\mathbb R}}
\newcommand{\PP}{{\mathbb P}}
\newcommand{\Pp}{{\mathbb P}}
\newcommand{\BF}{{\mathbb F}}
\newcommand{\QQ}{{\mathcal Q}}
\newcommand{\closure}[1]{\overline{#1}}
\newcommand{\EQ}{\Leftrightarrow}
\newcommand{\imply}{\Rightarrow}
\newcommand{\isom}{\cong}
\newcommand{\embed}{\hookrightarrow}
\newcommand{\tensor}{\mathop{\otimes}}
\newcommand{\wt}[1]{{\widetilde{#1}}}
\newcommand{\ol}{\overline}
\newcommand{\ul}{\underline}

\newcommand{\bs}{{\backslash}}
\newcommand{\CS}{{\mathcal S}}
\newcommand{\CA}{{\mathcal A}}
\newcommand{\Q}{{\mathbb Q}}
\newcommand{\F}{{\mathcal F}}
\newcommand{\sing}{{\text{sing}}}
\newcommand{\U} {{\mathcal U}}
\newcommand{\B}{{\mathcal B}}
\newcommand{\X}{{\mathcal X}}
\newcommand{\V}{{\mathcal V}}

\newcommand{\ECS}[1]{E_{#1}(X)}
\newcommand{\CV}[2]{{\mathcal C}_{#1,#2}(X)}

\newcommand{\rank}{\mathop{\mathrm{rank}}\nolimits}
\newcommand{\codim}{\mathop{\mathrm{codim}}\nolimits}
\newcommand{\Ord}{\mathop{\mathrm{Ord}}\nolimits}
\newcommand{\Var}{\mathop{\mathrm{Var}}\nolimits}
\newcommand{\Ext}{\mathop{\mathrm{Ext}}\nolimits}
\newcommand{\EXT}{\mathop{{\mathcal E}\mathrm{xt}}\nolimits}
\newcommand{\Pic}{\mathop{\mathrm{Pic}}\nolimits}
\newcommand{\Spec}{\mathop{\mathrm{Spec}}\nolimits}
\newcommand{\Jac}{\mathop{\mathrm{Jac}}\nolimits}
\newcommand{\Div}{\mathop{\mathrm{Div}}\nolimits}
\newcommand{\sgn}{\mathop{\mathrm{sgn}}\nolimits}
\newcommand{\supp}{\mathop{\mathrm{supp}}\nolimits}
\newcommand{\Hom}{\mathop{\mathrm{Hom}}\nolimits}
\newcommand{\Sym}{\mathop{\mathrm{Sym}}\nolimits}
\newcommand{\nilrad}{\mathop{\mathrm{nilrad}}\nolimits}
\newcommand{\Ann}{\mathop{\mathrm{Ann}}\nolimits}
\newcommand{\Proj}{\mathop{\mathrm{Proj}}\nolimits}
\newcommand{\mult}{\mathop{\mathrm{mult}}\nolimits}
\newcommand{\Bs}{\mathop{\mathrm{Bs}}\nolimits}
\newcommand{\Span}{\mathop{\mathrm{Span}}\nolimits}
\newcommand{\IM}{\mathop{\mathrm{Im}}\nolimits}
\newcommand{\Hol}{\mathop{\mathrm{Hol}}\nolimits}
\newcommand{\End}{\mathop{\mathrm{End}}\nolimits}
\newcommand{\CH}{\mathop{\mathrm{CH}}\nolimits}
\newcommand{\Exec}{\mathop{\mathrm{Exec}}\nolimits}
\newcommand{\SPAN}{\mathop{\mathrm{span}}\nolimits}
\newcommand{\birat}{\mathop{\mathrm{birat}}\nolimits}
\newcommand{\cl}{\mathop{\mathrm{cl}}\nolimits}
\newcommand{\rat}{\mathop{\mathrm{rat}}\nolimits}
\newcommand{\Bir}{\mathop{\mathrm{Bir}}\nolimits}
\newcommand{\Rat}{\mathop{\mathrm{Rat}}\nolimits}
\newcommand{\aut}{\mathop{\mathrm{aut}}\nolimits}
\newcommand{\Aut}{\mathop{\mathrm{Aut}}\nolimits}
\newcommand{\eff}{\mathop{\mathrm{eff}}\nolimits}
\newcommand{\nef}{\mathop{\mathrm{nef}}\nolimits}
\newcommand{\amp}{\mathop{\mathrm{amp}}\nolimits}
\newcommand{\DIV}{\mathop{\mathrm{Div}}\nolimits}
\newcommand{\Bl}{\mathop{\mathrm{Bl}}\nolimits}
\newcommand{\Cox}{\mathop{\mathrm{Cox}}\nolimits}
\newcommand{\NE}{\mathop{\mathrm{NE}}\nolimits}
\newcommand{\NM}{\mathop{\mathrm{NM}}\nolimits}
\newcommand{\Gal}{\mathop{\mathrm{Gal}}\nolimits}
\newcommand{\coker}{\mathop{\mathrm{coker}}\nolimits}
\newcommand{\ch}{\mathop{\mathrm{ch}}\nolimits}

\title{Nodal Curves on K3 Surfaces}

\author{Xi Chen}
\address{632 Central Academic Building\\
University of Alberta\\
Edmonton, Alberta T6G 2G1, CANADA}
\email{xichen@math.ualberta.ca}

\date{August 6, 2017}

\thanks{Research partially supported by Discovery Grant 262265 from the Natural Sciences and Engineering Research Council of Canada.} 

 

\begin{abstract}
In this paper, we study the Severi variety $V_{L,g}$ of genus $g$ curves in $|L|$ on a general polarized K3 surface $(X,L)$. 
We show that the closure of every component of $V_{L,g}$ contains a
component of $V_{L,g-1}$. As a consequence, we see that the general members of every component of $V_{L,g}$ are nodal.
\end{abstract}

\maketitle

\section{Introduction}

It was proved that every complete linear system on
a very general polarized K3 surface
$(X, L)$ contains a nodal rational curve \cite{C1} and
furthermore every rational
curve in $|L|$ is nodal, i.e., has only nodes $xy=0$ as singularities \cite{C2}.
The purpose of this note is to prove an analogous
result on singular curves in $|L|$ of geometric genus $g>0$.

For a line bundle $A$ on a projective surface $X$, we use the notation
$V_{A,g}$ to denote the Severi varieties of integral curves of geometric
genus $g$ in the complete linear series $|A| = \PP H^0(A)$.
For a K3 surface $X$, it is well known that every component of $V_{A,g}$ has the expected dimension $g$. Furthermore, using theory of deformation of maps, one can show that $\nu: \widehat{C}\to X$ is an immersion for $\nu$ the normalization of
a general member $[C]\in V_{A,g}$ if $g > 0$ \cite[Chap. 3, Sec. B]{H-M}.

It was claimed that a general member of $V_{A,g}$ is nodal on every projective K3
surface $X$ and every $A\in \Pic(X)$ as long as $g > 0$
in \cite[Lemma 3.1]{C1}. However, as kindly pointed out to the
author by Edoardo Sernesi \cite[Sec. 3.3]{D-S}, the proof there is wrong. So this note provides a partial fix for this problem, albeit only
for singular curves in the primitive class $|L|$ on a general
polarized K3 surface $(X,L)$. Our main theorem is

\begin{thm}\label{ECK3THM000}
For a general polarized K3 surface $(X, L)$, every (irreducible) component of $\overline{V}_{L,g}$ contains a component of $V_{L,g-1}$ for all $1\le g\le p_a(L)$, 
where $\overline{V}_{L,g}$ is the closure of $V_{L,g}$ in $|L|$
and $p_a(L) = L^2/2 + 1$ is the arithmetic genus of $L$.
\end{thm}

Clearly, the above theorem, combining with the fact that every rational curve in $|L|$ is nodal \cite{C2}, implies the following corollary by induction:

\begin{cor}\label{ECK3CORMODULI}
For a general polarized K3 surface $(X, L)$,
the general members of every component of $V_{L,g}$ are nodal for all $0\le g \le p_a(L)$.
\end{cor}

It was proved in  \cite[Theorem 1.3, 5.3 and Remark 5.6]{KLM} that the general members of every component of $V_{L,g}$ are not trigonal for $g \ge 5$. Combining with \cite[Theorem B.4]{D-S}, it shows that the corollary holds for $5\le g\le p_a(L)$. Of course, we have settled it for all genus $g$ here. As an application, it shows that the genus $g$ Gromov-Witten invariant computed in \cite{B-L} is the same as the number of genus $g$ curves in $|L|$ passing through $g$ general points. 

A comprehensive treatment for $V_{mL,g}$ is planned in a future paper.

As another potential application of Theorem \ref{ECK3THM000}, we want to mention the conjecture of the irreducibility of universal Severi variety $\V_{L,g}$ on K3 surfaces:

\begin{conj}\label{ECK3CONJUNIV}
Let $\K_p$ be the moduli space of polarized K3 surfaces $(X,L)$ of genus $p = p_a(L)$ and let
\begin{equation}\label{ECK3E007}
\V_{L,g} = \{(X,L,C): (X,L)\in \K_p, C\in V_{L,g}\}
\end{equation}
be the universal Severi variety of genus $g$ curves in $|L|$ over $\K_p$.
Then $\V_{L,g}$ is irreducible.
\end{conj}

If we approach the conjecture along the line of argument of J. Harris for the irreducibility of Severi variety of plane curves \cite{H}, we need to establish two facts:
\begin{itemize}
\item Every component of $\overline{\V}_{L,g}$ contains a component of $\V_{L,0}$.
\item $\V_{L,0}$ is irreducible and the monodromy action on the $p$ nodes of a rational curve $C\in V_{L,0}$ is the full symmetric group $\Sigma_p$ as $(X,L,C)$ moves in $\V_{L,0}$.
\end{itemize}

The second fact comes easily for plane curves, while the establishment of the first fact is the focus of Harris' proof (see also \cite[Chap. 6, Sec. E]{H-M}). The situation for $\V_{L,g}$ is somewhat reversed at the moment: the first fact follows from our main theorem, while the difficulty lies in the second fact:

\begin{conj}\label{ECK3CONJUNIVRAT}
Let $\V_{L,0}$ be the universal Severi variety of rational curves in $|L|$ over the moduli space $\K_p$ of polarized K3 surfaces $(X,L)$ of genus $p$ and let
\begin{equation}\label{ECK3E008}
\begin{aligned}
\W_{L,0} = \big\{(X,L,C,s_1,s_2,...,s_p): 
&\ (X,L,C)\in \V_{L,0},\\
&\ C_\text{sing} = \{s_1,s_2,...,s_p\}
\big\}.
\end{aligned}
\end{equation}
Then $\W_{L,0}$ is irreducible.
\end{conj}

Our above discussion shows that Conjecture \ref{ECK3CONJUNIVRAT}
implies \ref{ECK3CONJUNIV}.

\subsection*{Conventions}

We work exclusively over $\BC$.
A K3 surface in this paper is always projective. A polarized K3 surface is a pair $(X, L)$, where $X$ is a K3 surface and $L$ is an indivisible ample line bundle on $X$.

\subsection*{Acknowledgment}

I am very grateful to Edoardo Sernesi for pointing out the above-mentioned gap in my paper \cite{C1}. I also want to thank Thomas Dedieu for bringing the paper \cite{KLM} to my attention.

\section{Proof of Theorem \ref{ECK3THM000}}

We start with the following observation:

\begin{prop}\label{ECK3PROP000}
Let $W$ be a component of $V_{L,g}$ for a polarized K3 surface $(X,L)$ with $\Pic(X) = \BZ$. The following are equivalent:
\begin{enumerate}
\item\label{ECK3PROP000ITEM1}
The closure $\overline{W}$ of $W$ in $|L|$ contains a component of $V_{L,g-1}$.
\item\label{ECK3PROP000ITEM2}
$\dim (\overline{W}\backslash W) = g-1$.
\item\label{ECK3PROP000ITEM3}
For a set $\sigma$ of $g-1$ general points on $X$, $W\cap \Lambda_\sigma$ is not projective (i.e. complete), where $\Lambda_\sigma\subset |L|$ is the locus of curves $C\in |L|$ passing through $\sigma$. 
\end{enumerate}
\end{prop}

\begin{proof}
(\ref{ECK3PROP000ITEM1}) $\Rightarrow$ (\ref{ECK3PROP000ITEM2}) is obvious. Since every curve in $|L|$ is integral, we have
\begin{equation}\label{ECK3E000440}
\overline{W}\backslash W \subset \bigcup_{i<g} V_{L,i}.
\end{equation}
And since $\dim V_{L,i}\le i$, we have (\ref{ECK3PROP000ITEM2}) $\Rightarrow$ (\ref{ECK3PROP000ITEM1}).

Let $\partial W = \overline{W}\backslash W$. Obviously, 
$\dim (\partial W \cap \Lambda_\sigma) = \dim \partial W - (g-1)$. Therefore,
(\ref{ECK3PROP000ITEM2}) $\Rightarrow$ (\ref{ECK3PROP000ITEM3}). On the other hand, if 
$W\cap \Lambda_\sigma$ is not complete, then there exists $C_\sigma\in \partial W$ passing through $\sigma$. Then $\dim \partial W \ge g-1$.
So (\ref{ECK3PROP000ITEM3}) $\Rightarrow$ (\ref{ECK3PROP000ITEM2}).
\end{proof}

So it suffices to show that $W\cap \Lambda_\sigma$ is not complete for every component $W$ of $V_{L,g}$. We prove this using a degeneration argument similar to the one in \cite{C2}. 
A general K3 surface can be specialized to a {\em Bryan-Leung} (BL) K3 surface $X_0$,
which is a K3 surface with Picard lattice
\begin{equation}\label{ECK3E000}
\begin{bmatrix}
-2 & 1\\
1 & 0
\end{bmatrix}.
\end{equation}
It can be polarized by the line bundle $C + mF$, where
$C$ and $F$ are the generators of $\Pic(X_0)$ satisfying $C^2 = -2$,
$CF=1$ and $F^2 = 0$.
A general polarized K3 surface of genus $m$ can be degenerated to $(X_0, C+mF)$.
Such $X_0$ has an elliptic fibration $X_0\to \PP^1$ with fibers in $|F|$.
For a general BL K3 surface $X_0$, there are exactly $24$ nodal fibers in $|F|$.
A key fact here is that every member
of $|C + mF|$ is ``completely'' reducible in the sense that it is a union of $C$ and $m$ fibers in $|F|$ (counted with multiplicities).

Let $X$ be a family of K3 surfaces of genus $m$ over a smooth quasi-projective 
curve $T$ such that $X_0$ is a general BL K3 surface for a point $0\in T$, $X_t$ are K3 surfaces of $\Pic(X_t) = \BZ$ for $t\ne 0$ and $L$ is a line bundle on $X$ with $L_0 = C+mF$. After a base change, there exists $W\subset \V_{L,g}$ flat over $T$ such that $W_t$ is a component of $V_{L_t,g}$ for all $t\ne 0$. Let $\sigma$ be a set of $g-1$ general sections of $X/T$. It suffices to prove that $W_t\cap \Lambda_\sigma$ is not projective for $t$ general.

By stable reduction, there exists a family $f: Y\to X$ of genus $g$ stable maps over a smooth surface $S$ with the commutative diagram
\begin{equation}\label{ECK3E000150}
\begin{tikzcd}
Y \ar{r}{f} \ar{d} & X\ar{d}\\
S \ar{r}{\pi} & T
\end{tikzcd}
\end{equation}
where $S$ is flat and projective over $T$, $f_* Y_s \in \overline{W}_t \cap \Lambda_\sigma$ on $X_t$ for all $s\in S_t$ and $t\in T$ and $S$ dominates $\overline{W} \cap \Lambda_\sigma$ via the map sending $s\to [f_* Y_s]$. In other words, $f: Y\to X$ is the stable reduction of the universal family over $\overline{W}$
such that $f: Y_s\to X$ is the normalization of a general member $G\in W_t$ passing through the $g-1$ points $\sigma(t)$ for $s\in S_t$ general and $t\ne 0$.

Let us consider the moduli map $\rho: S\to \overline{\m}_g\times T$ sending $s\to ([Y_s], \pi(s))$, where
$\overline{\m}_g$ is the moduli space of stable curves of genus $g$ with $\m_g$ its open subset parameterizing smooth curves. To show that $W_t\cap \Lambda_\sigma$ is not complete, it suffices to show that 
\begin{equation}\label{ECK3E000901}
\rho^{-1}(\Delta\times T) \cap S_t \ne \emptyset
\end{equation}
for $t\ne 0$, where $\Delta = \overline{\m}_g \backslash \m_g$ is the boundary divisor of $\overline{\m}_g$.

Let $F_1, F_2, ..., F_{g-1}\subset X_0$ be $g-1$ fibers in $|F|$ passing through the $g-1$ points $\sigma(0)$, respectively. Since $\sigma(0)$ are in general position, $F_1, F_2, ..., F_{g-1}$ are $g-1$ general fibers in $|F|$ and $\sigma(0)\cap C = \emptyset$.

For every $s\in S_0$, $f_* Y_s\in |C+mF|$ passes through $\sigma(0)$. Therefore, we must have
\begin{equation}\label{ECK3E000900}
f_* Y_s = C + m_1 F_1 + m F_2 + ... + m_{g-1} F_{g-1} + M_s
\end{equation}
for some $m_1, m_2, ..., m_{g-1}\in \BZ^+$. Since the curves in $W_t\cap \Lambda_\sigma$ cover $X_t$ for $t\ne 0$, $f$ is surjective. Hence
$f_* Y_s$ covers $X_0$ as $s$ moves in $S_0$. 
Therefore, $M_s$ contains a moving fiber in $|F|$. More precisely, there exists a component $\Gamma$ of $S_0$ such that $\cup_{s\in \Gamma} M_s = X_0$.

For a general point $s\in \Gamma$, $M_s$ contains a general fiber $F_s$ in $|F|$. Therefore, $Y_s$ has components $\widehat{F}_{1,s}, \widehat{F}_{2,s}, ..., \widehat{F}_{g-1,s}, \widehat{F}_s$ dominating $F_1, F_2, ..., F_{g-1}, F_s$, respectively. And since
$p_a(Y_s) = g$, $\widehat{F}_{1,s}, \widehat{F}_{2,s}, ..., \widehat{F}_{g-1,s}, \widehat{F}_s$
are all elliptic curves. Indeed, it is very easy to see that its moduli $[Y_s]$ in $\overline{\m}_g$
\begin{equation}\label{ECK3E000889}
[Y_s] = [\widehat{C}_s \cup \widehat{F}_{1,s}\cup \widehat{F}_{2,s}\cup ...\cup \widehat{F}_{g-1,s}\cup \widehat{F}_s]
\end{equation}
is a smooth rational curve $\widehat{C}_s$ with $g$ elliptic ``tails''
$\widehat{F}_{1,s}, \widehat{F}_{2,s}, ..., \widehat{F}_{g-1,s}, \widehat{F}_s$ attached to it,
where $\widehat{C}_s$ is the component of $Y_s$ dominating $C$. Of course,
when $g \le 2$, $\widehat{C}_s$ is contracted under the moduli map.

Note that $\widehat{F}_{1,s}, \widehat{F}_{2,s}, ..., \widehat{F}_{g-1,s}, \widehat{F}_s$ are isogenous to $F_1, F_2, ..., F_{g-1}, F_s$, respectively. As $s$ moves on $\Gamma$,
$F_s$ moves in $|F|$. So $\widehat{F}_s$ has varying moduli. 
This shows that $\rho$ maps $S$ generically finitely onto its image. That is,
\begin{equation}\label{ECK3E000903}
\dim \rho(S) = 2.
\end{equation}
Furthermore, when $F_s$ becomes one of $24$ nodal fibers in $|F|$, $\widehat{F}_s$ becomes a union of rational curves. Therefore, there exists $b\in \Gamma$ such that $\widehat{F}_b$ is a connected union of rational curves with normal crossings and $p_a(\widehat{F}_b) = 1$. The moduli $[Y_b]$ of $Y_b$ is thus a smooth rational curve with $g-1$ elliptic tails and one nodal rational curve attached to it. Consequently,
\begin{equation}\label{ECK3E001}
\rho(b) \in \Delta_0\times T
\end{equation}
where $\Delta_0$ is the component of $\Delta$ whose general points parameterize curves of genus $g-1$ with one node. Combining \eqref{ECK3E000903}, \eqref{ECK3E001} and the fact that $\Delta_0$ is $\BQ$-Cartier, we conclude that
\begin{equation}\label{ECK3E002}
\rho(S)\cap (\Delta_0\times T) \ne \emptyset \text{ has pure dimension } 1.
\end{equation}
Therefore, for every connected component $G$ of $\rho^{-1}(\Delta_0\times T)$, we have
\begin{equation}\label{ECK3E004}
\dim \rho(G) = 1.
\end{equation}

If $\rho^{-1}(\Delta_0\times T)\cap S_t\ne\emptyset$ for $t\ne 0$, then \eqref{ECK3E000901} follows and we are done. Otherwise, 
\begin{equation}\label{ECK3E003}
\rho^{-1}(\Delta_0\times T) \subset S_0.
\end{equation}
Let $G$ be the connected component of $\rho^{-1}(\Delta_0\times T)$ containing the point $b$. Then $G\subset S_0$ and $\dim \rho(G) = 1$.

Let $B$ be an irreducible component of $G$ passing through $b$. For $Y_b$, we have
\begin{equation}\label{ECK3E0006}
f_* Y_b = C + m_1 F_1 + m F_2 + ... + m_{g-1} F_{g-1} + M_b
\end{equation}
with $M_b$ supported on the union $F_\Sigma$ of $24$ nodal rational curves in $|F|$. Therefore,
for $s\in B$ general, $M_s$ must also be supported on $F_\Sigma$; otherwise, $M_s$ contains a general member $F_s$ of $|F|$, the moduli $[Y_s]$ of $Y_s$ is given by \eqref{ECK3E000889} and $[Y_s]\not\in \Delta_0$.
Consequently, $M_s \equiv M_b$ for all $s\in B$ and $\rho$ is constant on $B$.

For a component $Q$ of $G$ with $q\in B\cap Q \ne \emptyset$, the same argument shows that $M_s \equiv M_q$ is supported on $F_\Sigma$ for all $s\in Q$ and $\rho$ is constant on $Q$. And since $G$ is connected, we can use this argument to show that $\rho$ is constant on every component of $G$, i.e., constant on $G$. This contradicts \eqref{ECK3E004}.

\end{document}